\documentclass{article}
\pdfoutput=1
\usepackage[latin2]{inputenc}
\usepackage{t1enc}
\usepackage{anysize}
\usepackage{amsmath, amssymb, epsfig, theorem}

\newtheorem{theorem}{Theorem}[section]
\newtheorem{lemma}[theorem]{Lemma}

\newenvironment{proof}[1][Proof]{\begin{trivlist}
\item[\hskip \labelsep {\bfseries #1}]}{\end{trivlist}}
\newenvironment{definition}[1][Definition]{\begin{trivlist}
\item[\hskip \labelsep {\bfseries #1}]}{\end{trivlist}}
\newenvironment{conjecture}[1][Conjecture]{\begin{trivlist}
\item[\hskip \labelsep {\bfseries #1}]}{\end{trivlist}}
\newenvironment{remark}[1][Remark]{\begin{trivlist}
\item[\hskip \labelsep {\bfseries #1}]}{\end{trivlist}}
\newenvironment{notation}[1][Notation]{\begin{trivlist}
\item[\hskip \labelsep {\bfseries #1}]}{\end{trivlist}}

\newcommand{\qed}{\nobreak \ifvmode \relax \else
      \ifdim\lastskip<1.5em \hskip-\lastskip
      \hskip1.5em plus0em minus0.5em \fi \nobreak
      \vrule height0.75em width0.5em depth0.25em\fi}

\author{Péter Burcsi\thanks{The research of P. Burcsi is  supported by the  by the European Union and co-financed by the European Social Fund (grant agreement no. T\'AMOP 4.2.1/B-09/1/KMR-2010-0003). },
Dániel T. Nagy\\
Eötvos Loránd University, Budapest}
\title{The method of double chains for largest families with excluded subposets}

\begin{document}

\maketitle

\begin{abstract}
For a given finite poset $P$, $La(n,P)$ denotes the largest size of a family $\mathcal{F}$ of subsets of $[n]$ not containing $P$ as a weak subposet. We exactly determine $La(n,P)$ for infinitely many $P$ posets. These posets are built from seven base posets using two operations. For arbitrary posets, an upper bound is given for $La(n,P)$ depending on $|P|$ and the size of the longest chain in $P$. To prove these theorems we introduce a new method, counting the intersections of $\mathcal{F}$ with double chains, rather than chains.
\end{abstract}

\section{Introduction}

Let $[n]=\{1,2,\dots, n\}$ be a finite set. We investigate families $\mathcal{F}$ of subsets of $[n]$ avoiding certain configurations of inclusion.

\begin{definition}
Let $P$ be a finite poset, and $\mathcal{F}$ be a family of subsets of $[n]$. We say that $P$ is contained in $\mathcal{F}$ if there is an injective mapping $f:P\rightarrow \mathcal{F}$ satisfying $a<_p b \Rightarrow f(a)\subset f(b)$ for all $a,b\in P$. $\mathcal{F}$ is called $P$-free if $P$ is not contained in it.

Let $La(n,P)=\{\max |\mathcal{F}|~ \mid~ \mathcal{F} ~\text{contains no} ~P\}$
\end{definition}
Note that we do not want to find $P$ as an induced subposet, so the subsets of $\mathcal{F}$ can satisfy more inclusions than the elements of the poset $P$.

We are interested in determining $La(n,P)$ for as many posets as possible. The first theorem of this kind was proved by Sperner. Later it was generalized by Erdős.

\begin{theorem} [Sperner] {\rm \cite{Sper}}
Let $\mathcal{F}$ be a family of subsets of $[n]$, with no member of $\mathcal{F}$ being the subset of an other one. Then
\begin{equation}
|\mathcal{F}|\leq {n \choose \lfloor n/2 \rfloor}
\end{equation}
\end{theorem}

\begin{theorem}[Erdős] {\rm \cite{Erd}} \label{erdospath}
Let $\mathcal{F}$ be a family of subsets of $[n]$, with no $k+1$ members of $\mathcal{F}$ satisfying $A_1\subset A_2\subset\dots \subset A_{k+1}$ $(k\leq n)$. Then $|\mathcal{F}|$ is at most the sum of the $k$ biggest binomial coefficients belonging to $n$. The bound is sharp, since it can be achieved by choosing all subsets $F$ with
$\lfloor\frac{n-k+1}{2}\rfloor\leq |F| \leq \lfloor\frac{n+k-1}{2}\rfloor$.
\end{theorem}

Since choosing all the subsets with certain sizes near $n/2$ is the maximal family for many excluded posets, we use the following notation.

\begin{notation}
$\Sigma (n,m)=\displaystyle\sum_{i=\lfloor\frac{n-m+1}{2}\rfloor}^{\lfloor\frac{n+m-1}{2}\rfloor} {n \choose i}$ denotes the sum of the $m$ largest binomial coefficients belonging to $n$.
\end{notation}

Now we can reformulate Theorem \ref{erdospath}. Let $P_{k+1}$ be the path poset with $k+1$ elements. Then
\begin{equation}
La(n,P_{k+1})=\Sigma (n,k)
\end{equation}

We give here a proof of Theorem \ref{erdospath} to illustrate the chain method introduced by Lubell \cite{Lub}.

\begin{proof} (Theorem \ref{erdospath})
A chain is $n+1$ subsets of $[n]$ satisfying $L_0\subset L_1\subset L_2\subset\dots \subset L_n$ and $|L_i|=i$ for all $i=0,1,2,\dots n$. The number of chains is $n!$. We use double counting for the pairs $(C,F)$ where $C$ is a chain, $F\in C$ and $F \in \mathcal{F}$.

The number of chains going through some subset $F\in \mathcal{F}$ is $|F|!(n-|F|)!$. So the number of pairs is
\[ \sum_{F\in\mathcal{F}} |F|!(n-|F|)! \]
One chain can contain at most $k$ elements of $\mathcal{F}$, otherwise a $P_{k+1}$ poset would be formed. So the number of pairs is at most $k\cdot n!$. It implies
\begin{equation}
\sum_{F\in\mathcal{F}} |F|!(n-|F|)!\leq  k\cdot n!
\end{equation}
\begin{equation}
\sum_{F\in\mathcal{F}} \frac{1}{{n \choose |F|}} \leq k
\end{equation}
Fixing $|\mathcal{F}|$, the left side takes its minimum when we choose the subsets with sizes as near to $n/2$ as possible. Choosing all $\Sigma (n,k)$ subsets with sizes $\lfloor\frac{n-k+1}{2}\rfloor\leq |F| \leq \lfloor\frac{n+k-1}{2}\rfloor$, we have equality. So we have
\begin{equation}
La(n,P_{k+1})=\Sigma (n,k) \qed
\end{equation}
\end{proof}

$La(n,P)$ is determined asymptotically for many posets, but its exact value is known for very few $P$. (See \cite{KSurvey} and \cite{GLu})

\section{The method of double chains}
The main purpose of the present paper is to exactly determine $La(n,P)$ for some posets $P$. Our main tool is a modification of the the chain method, double chains are used rather than chains.
\begin{definition}
Let $C:L_0\subset L_1\subset L_2\subset\dots \subset L_n$ be a chain. The double chain assigned to $C$ is a set $D=\{L_0, L_1, \dots , L_n, M_1, M_2, \dots , M_{n-1}\}$, where $M_i=L_{i-1}\cup(L_{i+1}\backslash L_{i})$.

Note that $|M_i|=|L_i|=i$, \\
$i<j\Rightarrow L_i\subset L_j,~ L_i\subset M_j,~ M_i\subset L_j$ and $i+1<j \Rightarrow M_i\subset M_j$.

$\{L_0, L_1, \dots , L_n\}$ is called the primary line of $D$ and $\{M_1, M_2,\dots , M_{n-1}\}$ is the secondary line.

$\mathcal{D}$ denotes the set of all $n!$ double chains.
\end{definition}

\begin{figure}[h]
\begin{center}
\includegraphics[scale = 0.25] {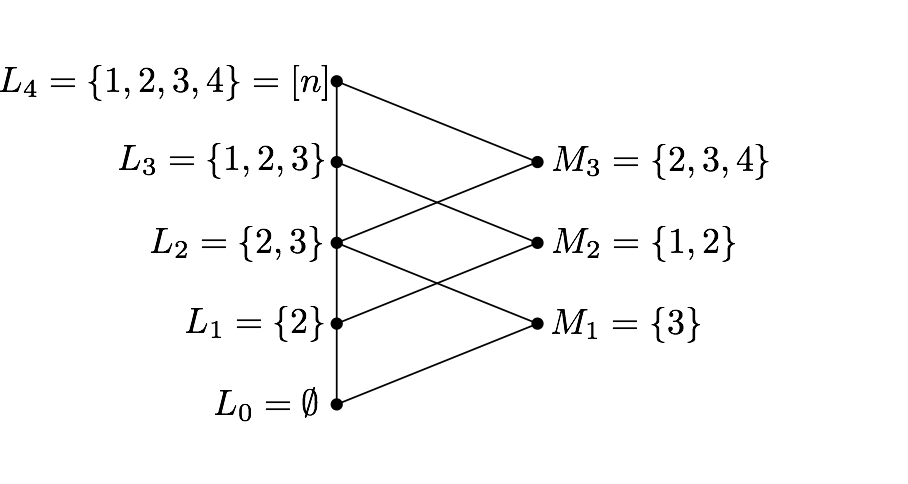}
\end{center}
\caption{The double chain assigned to the chain $\emptyset \subset \{2\} \subset \{2, 3\}
\subset \{1, 2, 3\} \subset \{1, 2, 3, 4\}$. }
\label{fig1}
\end{figure}

\begin{lemma}\label{dclamma1}
Let $\mathcal{F}$ be a family of subsets of $[n]$ $(n\geq 2)$, and let $m$ be a positive real number. Assume that
\begin{equation}
\sum_{D\in\mathcal{D}} |\mathcal{F}\cap D| \leq 2m\cdot n!
\end{equation}
Then
\begin{equation}\label{firstcase}
|\mathcal{F}|\leq m {n \choose \lfloor n/2 \rfloor}
\end{equation}
If $m$ is an integer and $m\leq n-1$, we have the following better bound:
\begin{equation}\label{secondcase}
|\mathcal{F}|\leq \Sigma(n,m)
\end{equation}
\end{lemma}

\begin{proof}
First we count how many double chains contains a given subset $F\subset [n]$. $\emptyset$ and $[n]$ are contained in all $n!$ double chains. Now let $F\not\in \{\emptyset, [n]\}$. $F$ is contained in the primary line of $|F|!(n-|F|)!$ double chains. Now count the double chains containing $F$ in the secondary line. Letting $F=M_{|F|}$, we have $|F|\cdot (n-|F|)$ possibilities to choose $L_{|F|}$, since we have to replace one element of $M_{|F|}$ with a new one. $M_{|F|}$ and $L_{|F|}$ already define $L_{|F|-1}$ and $L_{|F|+1}$. We have $(|F|-1)!$ and $(n-|F|-1)!$ possibilities for the first and last part of the primary line, so the number of double chains containing $F$ in the secondary line is $|F|(n-|F|)(|F|-1)!(n-|F|-1)!=|F|!(n-|F|)!$. It gives a total of $2|F|!(n-|F|)!$ double chains containing $F$.

Let $t=|\mathcal{F}\cap\{\emptyset, [n]\}|$. Double counting the pairs $(D,F)$ where $D\in \mathcal{D}$, $F\in D$ and $F\in \mathcal{F}$ we have
\begin{equation}
t\cdot n!+\sum_{F\in\mathcal{F}\backslash\{\emptyset, [n]\}} 2|F|!(n-|F|)! \leq 2m\cdot n!
\end{equation}
\begin{equation}\label{recbinom}
t\cdot\frac{1}{2}+\sum_{F\in\mathcal{F}\backslash\{\emptyset, [n]\}} \frac{1}{{n\choose |F|}} \leq m
\end{equation}

Since ${n \choose \lfloor n/2 \rfloor}$ is the biggest binomial coefficient, and ${n \choose \lfloor n/2 \rfloor}\geq 2$ we have
\begin{equation}
\frac{|\mathcal{F}|}{{n\choose \lfloor n/2 \rfloor}} \leq m
\end{equation}
It proves (\ref{firstcase}). If $m$ is an integer, and $m\leq n-1$, considering $|\mathcal{F}|$ fixed, the left side of (\ref{recbinom}) is minimal when we choose subsets with sizes as near to $n/2$ as possible. Choosing all $\Sigma (n,m)$ subsets with such sizes, we have equality in (\ref{recbinom}). It implies $|\mathcal{F}|\leq \Sigma(n,m)$, so (\ref{secondcase}) is proved. \qed
\end{proof}

\begin{definition}
The infinite double chain is an infinite poset with elements $L_i,~i\in \mathbb{Z}$ and $M_i,~i\in\mathbb{Z}$. The defining relations between the elements are
\[i<j\Rightarrow L_i\subset L_j,~ L_i\subset M_j,~ M_i\subset L_j\]
\end{definition}

\begin{figure}[h]
\begin{center}
\includegraphics[scale = 0.25] {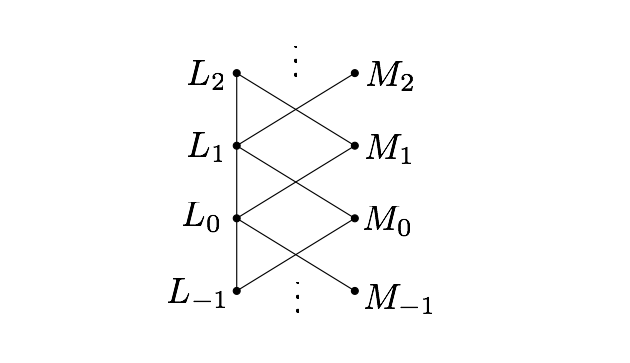}
\end{center}
\caption{The infinite double chain.}
\label{fig2}
\end{figure}

Note that the poset formed by the elements of any double chain with the inclusion as relation is a subposet of the infinite double chain.

\begin{lemma}\label{dclemma2}
Let $m$ be an integer or half of an integer and $P$ be a finite poset. Assume that any subset of size $2m+1$ of the infinite double chain contains $P$ as a (not necessarily induced) subposet. Let $\mathcal{F}$ be a family of subsets of $[n]$ such that $\mathcal{F}$ does not contain $P$. Then
\begin{equation}
|\mathcal{F}|\leq m {n \choose \lfloor n/2 \rfloor}
\end{equation}
If $m$ is an integer and $m\leq n-1$ we have the following better bound:
\begin{equation}
|\mathcal{F}|\leq \Sigma(n,m)
\end{equation}
\end{lemma}

\begin{proof}
Since the poset formed by the elements of any double chains is a subposet of the infinite double chain,
\\$|\mathcal{F}\cap D|\leq 2m$ for all double chains $D$. There are $n!$ double chains, so
\begin{equation}
\sum_{D\in\mathcal{D}} |\mathcal{F}\cap D| \leq 2m\cdot n!
\end{equation}
holds. Now we can use Lemma \ref{dclamma1} and finish the proof. \qed
\end{proof}

\section{An upper estimate for arbitrary posets}

\begin{definition}
The size of the longest chain in a finite poset $P$ is the largest integer $L(P)$ such that for some $a_1, a_2, \dots , a_{L(P)}\in P$, $a_1<_p a_2<_p\dots <_p a_{L(P)}$ holds.
\end{definition}

\begin{figure}[h]
\begin{center}
\includegraphics[scale = 0.5] {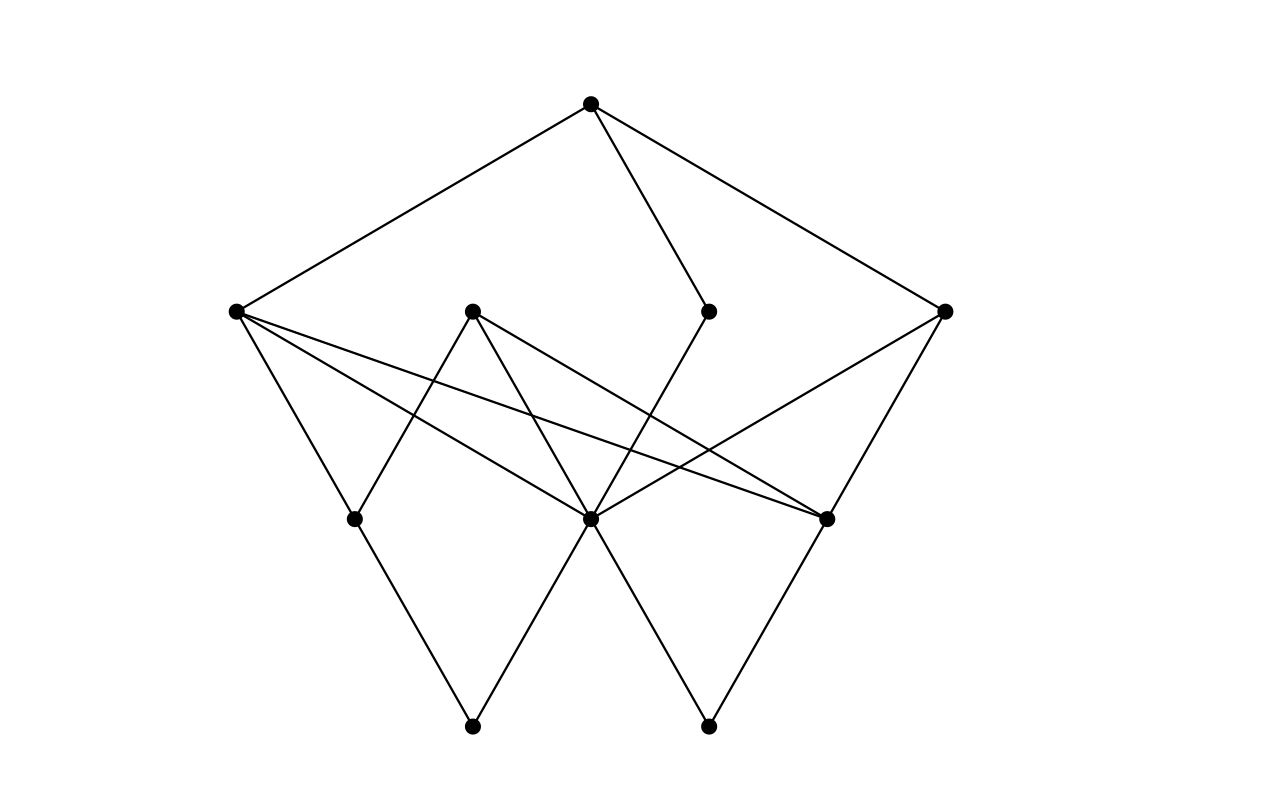}
\end{center}
\caption{A poset with $|P|=10$ elements and longest chain of length $L(P)=4$.}
\label{fig3}
\end{figure}

\begin{theorem}\label{mainbound}
Let $P$ be a finite poset and let $\mathcal{F}$ be a $P$-free family of subsets of $[n]$. Then
\begin{equation}
|\mathcal{F}|\leq \left(\frac{|P|+L(P)}{2}-1\right) {n \choose \lfloor n/2 \rfloor}
\end{equation}
If $\frac{|P|+L(P)}{2}-1$ is an integer and $\frac{|P|+L(P)}{2}\leq n$ we have the following better bound:
\begin{equation}
|\mathcal{F}|\leq \Sigma\left(n,\frac{|P|+L(P)}{2}-1\right)
\end{equation}
\end{theorem}

\begin{proof}
We want to use Lemma \ref{dclemma2} with $m=\frac{|P|+L(P)}{2}-1$. So the only thing we have to prove is the following lemma.
\end{proof}

\begin{lemma}
Let $P$ be a finite poset. Then any subset $S$ of size $|P|+L(P)-1$ of the infinite double chain contains $P$ as a (not necessarily induced) subposet.
\end{lemma}

\begin{proof}
We prove the lemma using induction on $L(P)$. When $L(P)=1$, we have a subset of size $|P|$ in the infinite double chain. We can choose them all, we get the poset $P$, since there are no relations between its elements. Assume that we already proved the lemma for all posets with longest chain of size $l-1$, and prove it for a poset $P$ with $L(P)=l$.

Arrange the elements of the infinite double chain as follows:
\[\dots L_{-1}, M_{-1}, L_0, M_0, L_1, M_1, L_2, M_2\dots  \]
Assume that $P$ has $k$ minimal elements, and choose the $k$ first elements of $S$ for them according to the above arrangement. Note that all remaining elements of $S$, except for at most one, are greater in the infinite double chain than all the $k$ elements we just chose. If there is such an exception, delete that element from $S$. Now we have at least $|P|+L(P)-k-2$ elements of $S$ left, all greater than the $k$ we chose for the minimal elements. Denote the set of these elements by $S'$.

Let $P'$ be the poset obtained by $P$ after deleting its minimal elements. It has $|P'|=|P|-k$ elements and a longest chain of size $L(P')=L(P)-1$. By the inductive hypothesis $P'$ is formed by some elements of $S'$, since $|S'|\geq |P|+L(P)-k-2 =|P'|+L(P')-1$. Considering these elements together with the first $k$, they form $P$ as a weak subposet in $S$. \qed
\end{proof}

\begin{remark}
The previously known upper bound for maximal families not containing a general $P$ as weak subposet was $\Sigma(n,|P|-1)$. We can get it from Theorem \ref{erdospath} since $P$ is a subposet of the path poset $P_{|P|}$. The new upper bound, $\Sigma\left(n,\frac{|P|+L(P)}{2}-1\right)$ is better since $L(P)\leq |P|$, and equality occurs only when $P$ is a path poset.
\end{remark}

\section{Exact results}

In this section we will describe infinitely many posets for which Theorem \ref{mainbound} provides a sharp bound.

\begin{definition}
For a finite poset $P$, $e(P)$ is the maximal $m$ such that the family formed by all subsets of $[n]$ of size $k, k+1, \dots , k+m-1$ is $P$-free for all $n$ and $k$.
\end{definition}

We will prove that $La(n,P)=\Sigma(n,e(P))$ if $n$ is large enough for infinitely many $P$, verifying the following conjecture for these posets.

\begin{conjecture} \cite{Bukh}
For all finite poset $P$
\begin{equation}
La(n,P)=e(p){n \choose \lfloor n/2 \rfloor}\left(1+O(1/n)\right)
\end{equation}
\end{conjecture}

In \cite{Bukh} Bukh proved the conjecture for all posets whose Hasse-diagram is a tree.

\begin{notation}
\begin{equation}\label{bp}
b(P)=\frac{|P|+L(P)}{2}-1 \text{, The bound used in Theorem \ref{mainbound}}
\end{equation}
\end{notation}

\begin{lemma} \label{eb}
Assume that $e(P)=b(P)$ for a finite poset $P$ and $n\geq b(P)+1$. Then
\begin{equation}
La(n,P)=\Sigma(n,e(P))=\Sigma(n,b(P))
\end{equation}
\end{lemma}

\begin{proof}
The family of subsets of size $\lfloor\frac{n-e(P)+1}{2}\rfloor \leq |F| \leq \lfloor\frac{n+e(P)-1}{2}\rfloor$ has $\Sigma(n,e(P))$ elements and is $P$-free by the definition of $e(P)$. On the other hand, Theorem \ref{mainbound} states that a $P$-free family has at most $\Sigma(n,b(P))$ elements. \qed
\end{proof}

Now we show some posets satisfying $e(P)=b(P)$.

\begin{definition} (See figure \ref{fig4}).
\\$E$ is the poset with one element.
\\The elements of the following posets are divided into levels so that $a$ is greater than $b$ in the poset if and only if $a$ is in a higher level than $b$.
\\$B$ is the butterfly poset, a poset with 2 elements on each level.
\\$D_3$ is the 3-diamond poset, a poset with respectively 1, 3 and 1 element on its levels.
\\$Q$ is a poset with respectively 2, 3 and 2 elements on its levels.
\\$R$ is a poset with respectively 1, 4, 4 and 1 element on its levels.
\\$S$ is a poset with respectively 1, 4 and 2 elements on its levels.
\\$S'$ is a poset with respectively 2, 4 and 1 element on its levels.
\end{definition}

\begin{figure}[h]
\begin{center}
\includegraphics[scale = 0.35] {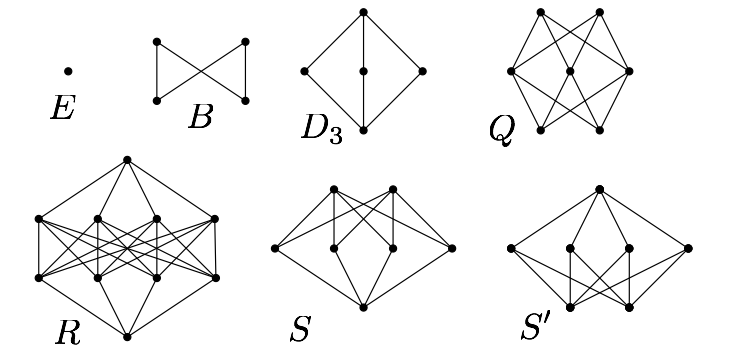}
\end{center}
\caption{7 small posets satisfying $e(P)=b(P)$.}
\label{fig4}
\end{figure}

\begin{lemma}\label{lotofposets}
For all $P\in \{E, B, D_3, Q, R, S, S'\}$, $e(P)=b(P)$ holds.
\end{lemma}

\begin{proof}
$b(P)$ is an integer for all the above posets. Assume that $e(P)\geq b(P)+1$. Then for $n\geq b(P)+1$ there would be a $P$-free family $\mathcal{F}$ of subsets of $[n]$ with $|\mathcal{F}|=\Sigma(n,b(P)+1)>\Sigma(n,b(P))$, contradicting Theorem \ref{mainbound}. So $e(p)\leq b(p)$. We will show that for every poset $P\in \{E, B, D_3, Q, R, S, S'\}$ and integers $n,k$ the family formed by all subsets of $[n]$ of size $k, k+1, \dots , k+b(P)-1$ is $P$-free. It gives us $e(P)\geq b(P)$, and completes the proof.

The statement is trivial for $P=E$ since $b(E)=0$.

$b(B)=2$. The set of all subsets with $k$ and $k+1$ elements is $B$-free since two subsets of size $k+1$ can not have two different common subsets of size $k$.

$b(D_3)=3$. The set of all subsets with $k, k+1$ and $k+2$ elements is $D_3$-free since for two subsets $A,B$,  $|B|-|A|\leq 2$ there are at most two subsets $F$ satisfying $A\subset F\subset B$.

$b(Q)=4$. Assume that $Q$ is formed by 7 subsets of size $k, k+1, k+2$ or $k+3$. There are at least 4 subsets in the lower 2 or the upper 2 levels. They should form a $B$ poset, that is not possible.

$b(R)=6$. Assume that $R$ is formed by 10 subsets of size $k, k+1, \dots , k+5$. Let $A$ be the least, and $B$ be the greatest subset. Let $U$ be the union of the 5 smaller subsets. At least 3 subsets in the second level are different from $U$, and contained in it. Similarly, at least 3 subsets of the third level are different from $U$,  and contain it. Since $D_3$ is not formed by subsets of size $m, m+1$ and $m+2$, $|A|+6\leq |U|+3\leq |B|$, a contradiction.

$b(S)=4$. Assume that $S$ is formed by 7 subsets of size $k, k+1, k+2$ or $k+3$. Let $V$ be the intersection of the two elements of the top level, then $|V|\leq k+2$. $V$ contains all elements of the middle level, and is different from at least 3 of them. These 3 elements together with the least element and $V$ form a $D_3$ from subsets of size $k,k+1$ and $k+2$, and it is a contradiction.

A family is $S'$-free if and only if the family of the complements of its elements is $S$-free. It gives $e(S')=e(S)\geq b(S)=b(S')$. \qed
\end{proof}

We define two ways of building posets from smaller ones, keeping the property $e(P)=b(P)$.

\begin{definition}
Let $P_1, P_2$ posets. $P_1 \oplus P_2$ is the poset obtained by $P_1$ and $P_2$ adding the relations $a<b$ for all $a\in P_1, b\in P_2$.

Assume that $P_1$ has a greatest element and $P_2$ has a least element. $P_1 \otimes P_2$ is the poset obtained by identifying the greatest element of $P_1$ with the least element of $P_2$.
\end{definition}

\begin{lemma}\label{eadd}
$e(P_1 \oplus P_2) \geq e(P_1) + e(P_2) + 1$. If $P_1\otimes P_2$ is defined, then $e(P_1 \otimes P_2) \geq e(P_1) + e(P_2)$.
\end{lemma}

\begin{proof}
In order to find a $P_1$, we need at least $e(P_1)+1$ levels, for a $P_2$, we need at least $e(P_2)+1$ levels. It follows from the properties of $\oplus$ that the lowest level of $P_2$ is above the highest level of $P_1$ in any occurrence of $P_1\oplus P_2$, which thus needs at least $e(P_1)+1+e(P_2)+1$ levels. In the case of $P_1\otimes P_2$, the same reasoning applies, noting that highest level of $P_1$ and the lowest level of $P_2$ coincide. \qed
\end{proof}

\begin{lemma}\label{ops}
Assume that $P_1$ and $P_2$ are finite posets such that $e(P_1)=b(P_1)$ and $e(P_2)=b(P_2)$. Then
\begin{equation}
e(P_1\oplus P_2)=b(P_1\oplus P_2)
\end{equation}
Assume that $P_1$ has a greatest element and $P_2$ has a least element. Then
\begin{equation}
e(P_1\otimes P_2)=b(P_1\otimes P_2)
\end{equation}
\end{lemma}

\begin{proof}
Note that $|P_1 \oplus P_2|=|P_1|+|P_2|$, $L(P_1\oplus P_2)=L(P_1)+L(P_2)$, and $e(P_1\oplus P_2)\geq e(P_1)+e(P_2)+1$. Similarly, $|P_1 \otimes P_2|=|P_1|+|P_2|-1$, $L(P_1\otimes P_2)=L(P_1)+L(P_2)-1$, and $e(P_1\otimes P_2)\geq e(P_1)+e(P_2)$.

From the above equations and (\ref{bp}) we have
\begin{equation}
e(P_1 \oplus P_2)\geq e(P_1)+e(P_2)+1=b(P_1)+b(P_2)+1=b(P_1 \oplus P_2)
\end{equation}
and
\begin{equation}
e(P_1 \otimes P_2)\geq e(P_1)+e(P_2)=b(P_1)+b(P_2)=b(P_1 \otimes P_2)
\end{equation}
if $P_1$ has a greatest element and $P_2$ has a least element.
We have already seen that $e(P) \leq b(P)$ always holds. \qed
\end{proof}

\begin{figure}[h]
\begin{center}
\includegraphics[scale = 0.35] {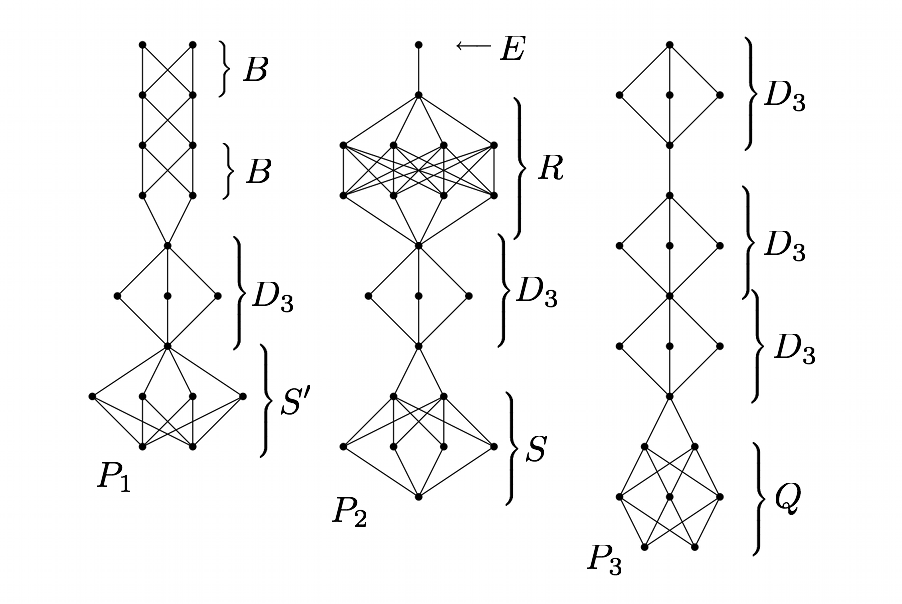}
\end{center}
\caption{Posets built from $E, B, D_3, Q, R, S$ and $S'$ using $\oplus$ and $\otimes$. $P_1 = S' \otimes D_3 \oplus B \oplus B$, $P_2 = S \oplus D_3 \otimes R \oplus E$ and $P_3 = Q \oplus D_3 \otimes D_3 \oplus D_3$.}
\label{fig5}
\end{figure}

The following theorem summarizes our results.

\begin{theorem}\label{main}
Let $P$ be a finite poset built from the posets $E, B, D_3, Q, R, S$ and $S'$ using the operations $\oplus$ and $\otimes$. (See figure \ref{fig5} for examples.) For $n\geq b(P)+1$
\begin{equation}
La(n,P)=\Sigma(n,b(P))=\Sigma(n, e(P))
\end{equation}
\end{theorem}

\begin{proof}
From Lemma \ref{lotofposets} and Lemma \ref{ops} we have $e(P)=b(P)$. Then Lemma \ref{eb} proves the theorem. \qed
\end{proof}

\begin{remark}
Theorem \ref{main} is the generalization of the theorem of Erdős (Theorem \ref{erdospath}), and the following two results.
\end{remark}

\begin{theorem}
[De Bonis, Katona, Swanepoel]{\rm \cite{DKS}} for $n\geq 3$
\begin{equation}
La(n,B)=\Sigma(n,2)
\end{equation}
\end{theorem}

\begin{theorem}
[Griggs, Li, Lu] {\rm(Special case of Theorem 2.5 in \cite{GLL})} for $n\geq 2$
\begin{equation}
La(n,D_3)=\Sigma(n,3)
\end{equation}
\end{theorem}

\end{document}